\documentclass[11pt]{amsart}

\usepackage{amsmath,amsthm,amsopn,amstext,amscd,amsfonts,amssymb}
\usepackage{graphics}
\usepackage{epsfig}
\usepackage{graphicx}
\usepackage[numbers,sort&compress]{natbib}
\bibpunct{[}{]}{,}{n}{}{;}
\theoremstyle{plain}
\newcommand{\CC}{\mathbb{C}}
\newcommand{\NN}{\mathbb{N}}
\newcommand{\osc}{{\rm Osc}}
\newcommand{\PP}{\mathbb{P}}
\newcommand{\RR}{\mathbb{R}}

\newcommand{\WW}{\mathcal{W}}

\newtheorem{teo}{\sc Theorem}[section]
\newtheorem{prop}{\sc Proposition}[section]

\newtheorem{coro}{\sc Corollary}[section]

\begin{document}

\title[On a Theorem by Bojanov and Naidenov applied to
 families of Gegenbauer-Sobolev polynomials]
{On a Theorem by Bojanov and Naidenov applied to
 families of Gegenbauer-Sobolev polynomials}

\author[V. Paschoa]{Vanessa G. Paschoa}
\address{Instituto de Ci\^{e}ncia e Tecnologia, Universidade Federal de S\~{a}o Paulo, Rua Talim 330, S\~{a}o Jos\'e dos Campos,
S\~{a}o Paulo, Brazil}
\email{vanessa.paschoa@unifesp.br}

\author[D. P\'erez]{Dilcia P\'erez$^{(1)}$}
\address{Departamento de Matem\'aticas, Universidad Centro-Occidental Lisandro Alvarado.
Barquisimeto, Lara-Venezuela}
\email{dperez@ucla.edu.ve}

\author[Y. Quintana]{Yamilet Quintana$^{(2)}$}
\address{Departamento de Matem\'aticas Puras y Aplicadas,
Edificio Matem\'aticas y Sistemas (MYS), Apartado Postal: 89000, Caracas 1080 A, Universidad Sim\'on Bol\'{\i}var, Venezuela}
\email{yquintana@usb.ve}
\thanks{$(1)$ Supported in part by a grant from   Consejo de Desarrollo Cient\'{\i}fico, Human\'{\i}stico y Tecnol\'ogico,  CDCHT-UCLA  (009-RCT-2012), Venezuela.}
\thanks{$(2)$ Supported in part by a grant from  Ministerio de Econom\'{\i}a y Competitividad, Direcci\'on General de Investigaci\'on Cient\'{\i}fica y T\'ecnica  (MTM2012-36732-C03-01), Spain.}

\date{December 03, 2013.\\ {\rm 2010 AMS Subject Classification.  33C45, 41A17.}}

%%%%%%%%%%%%%%%%%%%%%%%%%%%%%%%%%%

\begin{abstract}
Let $\{Q^{(\alpha)}_{n,\lambda}\}_{n\geq 0}$ be the sequence of monic orthogonal polynomials with respect the Gegenbauer-Sobolev inner product
$$\langle f,g\rangle_{S}:=\int_{-1}^{1}f(x)g(x)(1-x^{2})^{\alpha-\frac{1}{2}}dx+\lambda \int_{-1}^{1}f'(x)g'(x)(1-x^{2})^{\alpha-\frac{1}{2}} dx,$$
where $\alpha>-\frac{1}{2}$ and $\lambda\geq 0$. In this paper we use a recent result due to B.D. Bojanov and N.
 Naidenov \cite{BN2010}, in order to study the maximization of a local extremum of the $k$th derivative $\frac{d^k}{dx^k}Q^{(\alpha)}_{n,\lambda}$ in $[-M_{n,\lambda}, M_{n,\lambda}]$, where
 $M_{n,\lambda}$ is a suitable value such that all zeros of the polynomial $Q^{(\alpha)}_{n,\lambda}$ are contained in $[-M_{n,\lambda}, M_{n,\lambda}]$ and the function $\left|Q^{(\alpha)}_{n,\lambda}\right|$ attains its maximal value at the end-points of such interval. Also, some illustrative numerical examples are presented.\\

\hspace{-.6cm} {\it Key words and phrases:} Sobolev orthogonal polynomials; oscillating polynomials.\\

\end{abstract}

\maketitle{}

\markboth{V.G. Paschoa, D. P\'erez and Y. Quintana \mbox{}
}{On a Theorem by Bojanov and Naidenov applied to
 families of G-S polynomials}

%%%%%%%%%%%%%%%%%%%%%%%%%%%%%%%%%%

\section{Introduction}

Extremal  properties for general orthogonal polynomials is an interesting subject in
approximation theory and their applications permeate many fields in science and engineering
\cite{be1995,Milo,N79,ST92,Sze1}. Although it may seem an old subject from the view point of the standard
orthogonality \cite{be1995,Milo,Sze1}, this is not the case neither in the general setting (cf. \cite{LPP05,DL07,DL94,LL2001,N86}) nor from the view point of Sobolev
orthogonality, where it remains like a partially explored subject \cite{BMJL06}.
In fact, new results continue to appear in some recent publications
\cite{LPP01,LPP05,DL07,PQ2011,PQRT10,PQRT11}.

\medskip

Let $d\mu(x)=(1-x^{2})^{\alpha-\frac{1}{2}}dx$ with $\alpha>-\frac{1}{2}$, be the Gegenbauer measure supported on the interval $[-1,1]$. We consider the following Sobolev inner product on the linear space $\PP$ of polynomials with real coefficients.

\begin{equation}
\label{ec1}
\langle f,g\rangle_{S}:=\int_{-1}^{1}f(x)g(x)d\mu(x)+\lambda \int_{-1}^{1}f'(x)g'(x)d\mu(x),
\end{equation}
where $\lambda\geq 0$. Let $\{Q^{(\alpha)}_{n,\lambda}\}_{n\geq 0}$ denote the sequence of monic orthogonal polynomials with respect to \eqref{ec1}. These polynomials are usually called monic Gegenbauer-Sobolev polynomials \cite{fej1,fej2,MPP94,AMF2001,AMF97,PQW2001} and it is known that the zeros of these polynomials are in the real line \cite{MPP94,AMF2001}, and therefore they belong to other important class of algebraic polynomials, namely the oscillating polynomials \cite{BN2010,NN2005}.

\medskip

The main result of \cite{BN2010}  allows to guarantee the maximal absolute value of higher derivatives
of a symmetric oscillating polynomial on a finite interval are attained in the
end-points of such interval, whenever   the maximal absolute value of the polynomial is
attained in the end-points of that interval.  Then, \cite[Section 4]{BN2010} contains a brief explanation about  applications of previous result  to  orthogonal polynomials on the real line
associated to symmetric weights. We focus our attention  on that last part of \cite[Section 4]{BN2010} in order to enlarge the range of
application of \cite[Theorem 1]{BN2010} to the class of  Gegenbauer-Sobolev polynomials corresponding to the inner product \eqref{ec1}.

\medskip

The paper  is structured as follows. Section 2 provides some background about structural properties of the
 Gegenbauer and Gegenbauer-Sobolev polynomials corresponding to the inner product \eqref{ec1}, respectively. Section 3 contains some well-known characteristics of the class of oscillating polynomials on a finite interval. We also state there our main result (see Theorem \ref{t:main}) and provide some illustrative numerical examples. Throughout this paper, the notation $u_{n}\cong v_{n}$ means that the
sequence $\left\{ \frac{u_{n}}{v_{n}}\right\} _{n}$ converges to $1$ as $n\rightarrow \infty $. We will denote by $\PP_{n}$ and $\|f\|_{\infty}$, the space of polynomials of degree at most $n$ and the uniform norm of $f$ on the interval in consideration, respectively. Any other standard notation will be properly introduced whenever needed.

\section{Basic facts: Gegenbauer and Gegenbauer-Sobolev orthogonal polynomials}

For $\alpha>-\frac{1}{2}$  we denote by $\{\hat{C}^{(\alpha)}_{n}\}_{n\geq 0}$ the sequence of Gegenbauer polynomials, orthogonal on $[-1,1]$ with respect to the measure $d\mu(x)$ (cf. \cite[Chapter IV]{Sze1}), normalized by
$$\hat{C}^{(\alpha)}_{n}(1)=\frac{\Gamma(n+2\alpha)}{n!\Gamma(2\alpha)}.$$

It is clear that this normalization does not allow $\alpha$ to be zero or a negative integer. Nevertheless, the following limits exist for every $x\in[-1,1]$ (see \cite[formula (4.7.8)]{Sze1}.)

$$\lim_{\alpha\to 0}\hat{C}^{(\alpha)}_{0}(x)=T_{0}(x), \quad \lim_{\alpha\to 0}\frac{\hat{C}^{(\alpha)}_{n}(x)}{\alpha}=\frac{2}{n}T_{n}(x),$$
where $T_n$ is the $n$th Chebyshev polynomial of the first kind. In order to avoid confusing notation, we define the sequence  $\{\hat{C}^{(0)}_{n}\}_{n\geq 0}$ as follows.

$$\hat{C}^{(0)}_{0}(1)=1, \quad  \hat{C}^{(0)}_{n}(1)=\frac{2}{n}, \quad \hat{C}^{(0)}_{n}(x)=\frac{2}{n}T_{n}(x), \quad n\geq 1.$$

We denote the $n$th monic Gegenbauer orthogonal polynomial by
\begin{equation}
\label{mgen1}
C^{(\alpha)}_{n}(x)= (h^{\alpha}_{n})^{-1}\hat{C}^{(\alpha)}_{n}(x),
\end{equation}
where the constant $h^{\alpha}_{n}$  (cf. \cite[formula (4.7.31)]{Sze1}) is given by
\begin{eqnarray}
h^{\alpha}_{n}&=& \frac{2^{n}\Gamma(n+\alpha)}{n!\Gamma(\alpha)}, \quad \alpha\not=0,\\
h^{0}_{n}&=& \lim_{\alpha\to 0}\frac{h^{\alpha}_{n}}{\alpha}=\frac{2^{n}}{n}, \quad n \geq 1.
\end{eqnarray}

Then for $n\geq 1$, we have  $C^{(0)}_{n}(x)= \lim_{\alpha\to 0} (h^{\alpha}_{n})^{-1}\hat{C}^{(\alpha)}_{n}(x)= \frac{1}{2^{n-1}}T_{n}(x)$.

\medskip

\begin{prop}
\label{prop1}
Let  $\{C^{(\alpha)}_{n}\}_{n\geq 0}$ be the sequence of monic Gegenbauer orthogonal polynomials. Then the following statements hold.
\begin{enumerate}
\item Three-term recurrence relation. For every $n\in \NN$,
\begin{equation}
xC^{(\alpha)}_{n}(x)=C^{(\alpha)}_{n+1}(x)+\gamma_{n}^{(\alpha)}C^{(\alpha)}_{n-1}(x), \quad \alpha>-\frac{1}{2},\quad \alpha\not=0,
\label{rr3t}
\end{equation}
with initial conditions $C^{(\alpha)}_{0}(x)=1$, $C^{(\alpha)}_{1}(x)=x$, and recurrence coefficient $\gamma _{n}^{(\alpha)}=\frac{n(n+2\alpha-1)}{4(n+\alpha)(n+\alpha-1)}$.

\item For every $n\in \NN$ (see \cite[formula (4.7.15)]{Sze1}),
\begin{equation}
\label{GegNorm}
\|C^{(\alpha)}_{n}\|_{\mu}^{2}=\int_{-1}^{1}[C^{(\alpha)}_{n}(x)]^{2}d\mu(x)=\pi 2^{1-2\alpha-2n}\frac{n!\Gamma(n+2\alpha)}{\Gamma(n+\alpha+1)\Gamma(n+\alpha)}.
\end{equation}

\item Structure relation (see \cite[formula (4.7.29)]{Sze1}). For  every $n\geq 2$,
\begin{equation}
C^{(\alpha-1)}_{n}(x)=C^{(\alpha)}_{n}(x)+\xi_{n-2}^{(\alpha)}C^{(\alpha)}_{n-2}(x),
\label{strre}
\end{equation}
where
\begin{equation}
\label{strre1}
\xi_{n}^{(\alpha)}= \frac{(n+2)(n+1)}{4(n+\alpha+1)(n+\alpha)}, \quad n\geq 0.
\end{equation}

\item For every $n\in \NN$ (see \cite[formula (4.7.14)]{Sze1}),
\begin{equation}
\label{GegDer}
\frac{d}{dx} C^{(\alpha)}_{n}(x)=n C^{(\alpha+1)}_{n-1}(x).
\end{equation}
\end{enumerate}
\end{prop}

Some well-known properties of the monic Gegenbauer-Sobolev orthogonal polynomials corresponding to the inner product \eqref{ec1} are the following.

\begin{prop}
\label{prop2}
Let $\{Q^{(\alpha)}_{n,\lambda}\}_{n\geq 0}$ be the sequence of monic orthogonal polynomials with respect to \eqref{ec1}. Then the following statements hold.
\begin{enumerate}
\item The polynomials $Q^{(\alpha)}_{n,\lambda}$ are symmetric, i.e.,
\begin{equation}
\label{symm1}
Q^{(\alpha)}_{n,\lambda}(-x)=(-1)^{n}Q^{(\alpha)}_{n,\lambda}(x).
\end{equation}
\item The zeros of $Q^{(\alpha)}_{n,\lambda}$  are real and simple, and  they interlace with the zeros of the monic Gegenbauer orthogonal polynomials  $C^{(\alpha)}_{n}$. Furthermore, for $\alpha\geq \frac{1}{2}$ they are all contained in the interval $[- 1, 1]$ and for $-\frac{1}{2}<\alpha<\frac{1}{2}$  there is at most a pair of zeros symmetric with respect to the origin outside the interval
$[-1, 1]$, (cf. \cite{MPP94,AMF2001}).
\item \cite[Lemma 5.1]{MPP94}. For $\alpha\geq \frac{1}{2}$, we have $Q^{(\alpha)}_{n,\lambda}(1)>0$.
\end{enumerate}
\end{prop}

It is worthwhile to point out that in the case $-\frac{1}{2}<\alpha<\frac{1}{2}$, no global properties about the sign $Q^{(\alpha)}_{n,\lambda}(1)$ can be deduced (cf. \cite{MPP94}.)

\medskip

However,  the location of zeros of Sobolev orthogonal polynomials is not a trivial problem. For instance, if we consider $(\mu_{0},\mu_{1})$  a vector of compactly supported positive measures on the real line
with finite total mass and the following Sobolev inner product on the linear space $\PP$ of polynomials with real coefficients.

\begin{equation}
\label{mark7}
\langle f,g\rangle_{(\mu_{0},\mu_{1})}:=\int f(x)g(x)d\mu_{0}(x)+\int f'(x)g'(x)d\mu_{1}(x),
\end{equation}
then, simple examples show that the zeros of these Sobolev orthogonal polynomials  do not necessarily remain in the convex hull of the union of the
supports of the measures $\mu_k$, $k=0,1$, and  they can be complex. In this  regard some numerical experiments may be found in  \cite{GK97}.  In particular, the boundedness of the zeros of Sobolev
orthogonal polynomials is an open problem \cite{BMJL06,AMF2001}, but as was stated in \cite{LPP01}, it could be obtained as a consequence of the boundedness of the multiplication operator $Mf (z) = z f (z)$: If $M$ is bounded and $\|M\|$ is its operator norm (induced by \eqref{mark7}), then all the zeros of the Sobolev orthogonal polynomials $Q_n$ are contained in the disc $\{z\in \CC:|z|\leq \|M\|\}$.

\medskip

Indeed, if $x_0$ is a zero of $Q_n$ then $xp(x) = x_{0}p(x) + Q_n (x)$ for a polynomial $p \in \PP_{n-1}$. Since $p$ and $Q_n$ are orthogonal, we get
$$|x_{0}|^{2} \|p\|_{(\mu_{0},\mu_{1})}^{2}= \|xp\|^{2}_{(\mu_{0},\mu_{1})}- \|Q_n \|^{2}_{(\mu_{0},\mu_{1})}\leq \|xp\|_{(\mu_{0},\mu_{1})}^{2} = \|Mp\|_{(\mu_{0},\mu_{1})}^{2}\leq \|M\|^{2}\|p\|_{(\mu_{0},\mu_{1})}^{2},$$
which yields the above result.

\medskip

Thus, in the last decades the question whether or not the multiplication operator $M$ is bounded
has been a topic of interest to investigators in the field, since it turns out to be a key
for the location of zeros and for establishing  the asymptotic behavior of orthogonal polynomials with
respect to diagonal (or non-diagonal) Sobolev inner products (cf. \cite{AMF2001,PQRT10,PQRT11} and the references therein)

\medskip

From the structure relation \eqref{strre} and  \cite[formula (3)]{AMF97} (see also \cite[Proposition 1]{fej1}) the following connection formula can be obtained.

\begin{prop}
\label{cfor}
For $\alpha>-\frac{1}{2}$,
\begin{equation}
\label{conec}
C^{(\alpha-1)}_{n}(x)=Q^{(\alpha)}_{n,\lambda}(x)-d_{n-2}(\alpha)Q^{(\alpha)}_{n-2,\lambda}(x), \quad n\geq 2,
\end{equation}
where
\begin{equation}
\label{conec1}
d_{n}(\alpha)= \xi^{(\alpha)}_{n}\frac{\|C^{(\alpha)}_{n} \|_{\mu}^{2}}{\|Q^{(\alpha)}_{n,\lambda}\|_{S}^{2}}.
\end{equation}
Moreover,
\begin{equation}
\label{conec2}
d_{n}(\alpha) \cong \frac{1}{16\lambda n^2}.
\end{equation}
\end{prop}

\section{Maximization of local extremum of the derivatives for families of Gegenbauer-Sobolev polynomials}

A polynomial $P\in\PP$ is said oscillating (see \citep{B79,BN2010,BR95,NN2005,GN2001,GN2005}) if it has all its zeros on the real line $\RR$. For example, the classical orthogonal polynomials on subsets of $\RR$ (Hermite, Laguerre and Jacobi polynomials \cite{dimitar2012,N86,Sze1}), orthogonal polynomials for weights in the Nevai class $M(0,1)$ \cite{N79}, including whose orthogonal with respect to weights belonging to Levin-Lubinsky class $\hat{\WW}$ \cite{DL94}, and a broad class of Sobolev orthogonal polynomials \citep{fej1,GK97,MPP94,AMF2001,AMF97,PQW2001} constitute an important family of oscillating polynomials. Usually, when all zeros of a polynomial $P\in\PP_{n}$ with $\deg(P)=n$,  are contained in a given finite interval $[a,b]$, it is called oscillating polynomial on $[a,b]$, (see \citep{BN2010,NN2005}.)

\medskip

We denote by $\osc(\RR)$ and $\osc[a,b]$ the classes of oscillating polynomials on $\RR$ and $[a,b]$, respectively. For any $P\in \osc[a,b]$ with $\deg(P)=n$, we consider  the vector ${\bf h}(P)=(h_{0}(P),\ldots, h_{n}(P))$, where $h_{j}(P)=|P(t_{j})|$, $0\leq j\leq n$, $t_{0}=a$,  $t_{n}=b$, and  $t_{1}\leq t_{2}\leq\cdots\leq t_{n-1}$ are the zeros of $P'$.

\medskip

Amongst the main characteristics of the class  $\osc[a,b]$ we list the following.
\begin{enumerate}
\item[i)] $P'\in  \osc[a,b]$, for all $P\in  \osc[a,b]$.
\item[ii)] Any $P\in  \osc[a,b]$ is completely determined by its local extrema and the values at the end-points of the interval $[a,b]$ (cf. \cite[Theorem 2]{B79}, \cite[Remark 1]{BR95}.)
\item[iii)] For $P\in \osc[a,b]$ with $\deg(P)=n$, there exists a monotone dependence of the parameters $h_{j}(P')$ on the parameters $h_{0}(P),\ldots, h_{n}(P)$ of $P$ (cf. \cite[Lemma 3]{BR95}.)
\item[iv)] If $P\in \osc[a,b]$ with $\deg(P)\geq 3$ and $\|P\|=|P(a)|$, then each local extremum of $P'$ from the first half (i.e.,  with an index less than or equal to $\left\lfloor \frac{n-1}{2}\right\rfloor$, and $\lfloor t \rfloor $ denoting the integer part of $t$) is smaller in absolute value than $|P'(a)|$.
\end{enumerate}

More precisely, the property iv) was stated in the following theorem.

\begin{teo}
\label{tC}
(\cite[Theorem 1]{BN2010}) Let $P\in \osc[a,b]$ with $\deg(P)\geq 3$. Assume that $\|P\|_{\infty}=|P(a)|=1$. Then
\begin{equation}
\label{eq3}
|P'(\tau_{j})|<|P'(a)|, \mbox{ for } j=0,\ldots, \left\lfloor \frac{n-1}{2}\right\rfloor,
\end{equation}
where $\tau_{1}\leq \cdots\leq \tau_{n-2}$ are the zeros of $P''$.
\end{teo}

\begin{coro}
\label{cA}
(\cite[Corollary 1]{BN2010}) Let $P\in \osc[-1,1]$ be a symmetric polynomial, with $\deg(P)=n$. Assume that $\|P\|_{\infty}=|P(1)|=1$. Then,
\begin{equation}
\label{eq4}
\|P^{(k)}\|_{\infty}=|P^{(k)}(1)|, \, \mbox{ for } k=1,\ldots, n.
\end{equation}
\end{coro}

\medskip

As a consequence of the combination of Theorem \ref{tC}  (or Corollary \ref{cA}) and the structural properties of the sequence $\{Q^{(\alpha)}_{n,\lambda}\}_{n\geq 0}$ given in the previous section, we can obtain the maximization of local extremum of the derivatives for the sequence  $\{Q^{(\alpha)}_{n,\lambda}\}_{n\geq 0}$  as follows.

\medskip

Let $\{Q^{(\alpha)}_{n,\lambda}\}_{n\geq 0}$ be the sequence of monic orthogonal polynomials with respect to \eqref{ec1}. Let us consider   $x_{n,\lambda}^{\alpha, [1]}<x_{n,\lambda}^{\alpha, [2]}<\cdots< x_{n,\lambda}^{\alpha, [n]}$ the zeros of the Gegenbauer-Sobolev polynomial $Q^{(\alpha)}_{n,\lambda}$ and $N$  the maximum value attained by $|Q^{(\alpha)}_{n,\lambda}(x)|$ on the interval $[x_{n,\lambda}^{\alpha, [1]},x_{n,\lambda}^{\alpha, [n]}]$. Then $M_{n,\lambda}$ can be defined  to be the minimal real point such that $ x_{n,\lambda}^{\alpha, [n]}<M_{n,\lambda}$ and $|Q^{(\alpha)}_{n,\lambda}(M_{n,\lambda})|= N$, i.e., $M_{n,\lambda}$ is the point closest to $x_{n,\lambda}^{\alpha, [n]}$ where the maximal absolute value of the polynomial  $Q^{(\alpha)}_{n,\lambda}$ is attained. Notice that $M_{n,\lambda}$ also depends on the parameter $\alpha$ and  $Q^{(\alpha)}_{n,\lambda}\in\osc[-M_{n,\lambda},M_{n,\lambda}]$. Thus,  we can consider the following normalized polynomials

\begin{equation}
\label{poli2}
q^{(\alpha)}_{n,\lambda}(x):= \frac{Q^{(\alpha)}_{n,\lambda}(x)}{Q^{(\alpha)}_{n,\lambda}(M_{n,\lambda})}, \quad x\in [-M_{n,\lambda},M_{n,\lambda}], \quad n\geq 0.
\end{equation}

\begin{teo}
\label{t:main}
Let $\{q^{(\alpha)}_{n,\lambda}\}_{n\geq 0}$   be the sequence of orthogonal polynomials given in
 \eqref{poli2}. Then $\left|\frac{d^k}{dx^k}q^{(\alpha)}_{n,\lambda}\right|$ attains its maximal value on the interval $[-M_{n,\lambda},M_{n,\lambda}]$ at the end-points, for  $\alpha>-\frac{1}{2}$ and $1\leq k\leq n$.
\end{teo}

\begin{proof}
It  suffices  to follow the proof of Theorem \ref{tC} (or Corollary \ref{cA}) given in \cite[Theorem 1 (or Corollary 1)]{BN2010}  by making the corresponding modifications.
\end{proof}

\medskip

Notice that from a numerical point of view the value $M_{n,\lambda}$ can be difficult to obtain for $n$ large enough. However, for any value $K>0$ such that $N<|Q^{(\alpha)}_{n,\lambda}(x)|$ for $x < -K$ and $ x> K$, the result of Theorem \ref{t:main} remains true on the interval  $[-K,K]$.

\medskip

We finish this section providing some illustrative numerical examples (with the help of MAPLE) about the
above result for different values of $n$, $\alpha$ and $\lambda$ (see Figure 1 and Figure 2 below).

\vspace{-4cm}

\begin{figure}[h]
\vspace*{2.5cm}
    \centering
        \includegraphics[totalheight=5in]{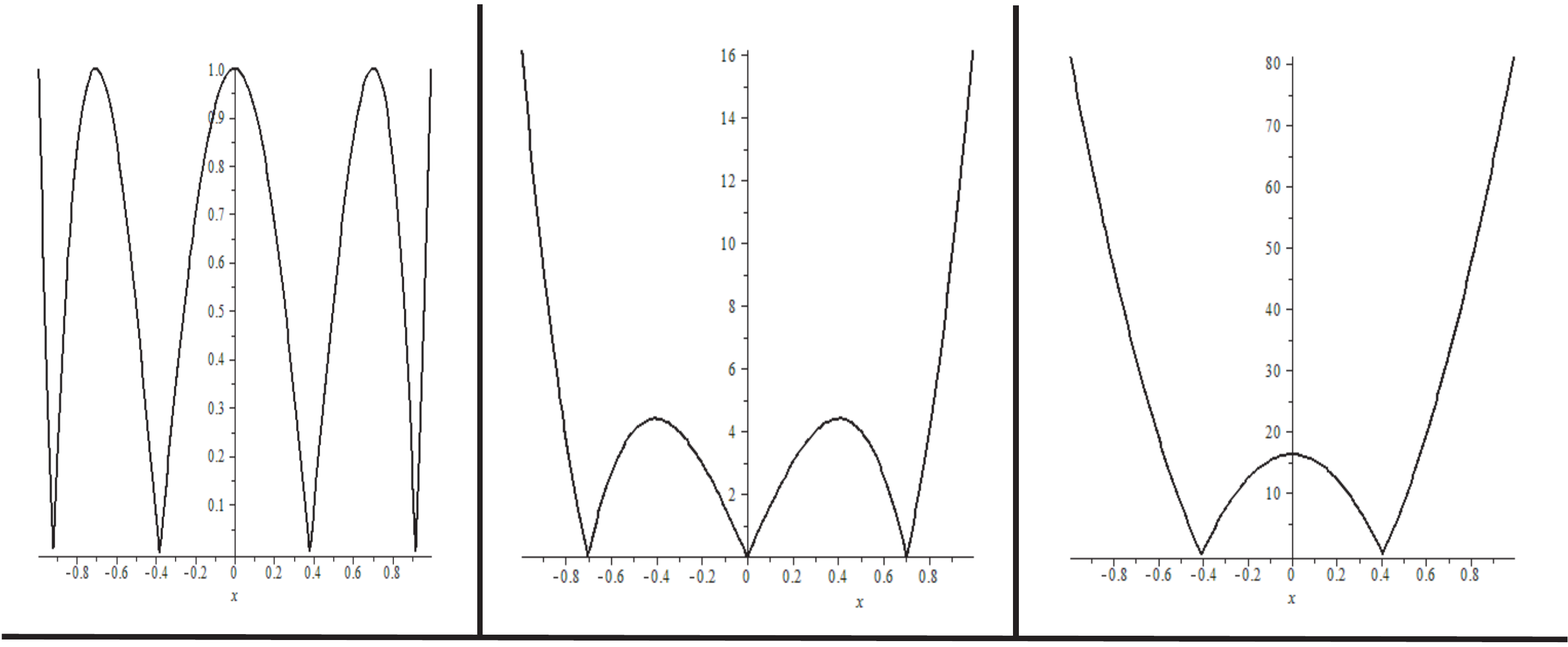}
        \vspace*{-3cm}
   \caption{Graphics of $|\frac{d^k}{dx^k}q^{(\alpha)}_{n,\lambda}|$ for $n=4$, $\alpha=\lambda=1$, $M_{n,\lambda}=0.9926198253$  and $k=0,1,2$, respectively.}
\label{figur1}
\end{figure}

\newpage

\vspace{-3cm}

\begin{figure}[h]
\vspace*{2.5cm}
    \centering
        \includegraphics[totalheight=5in]{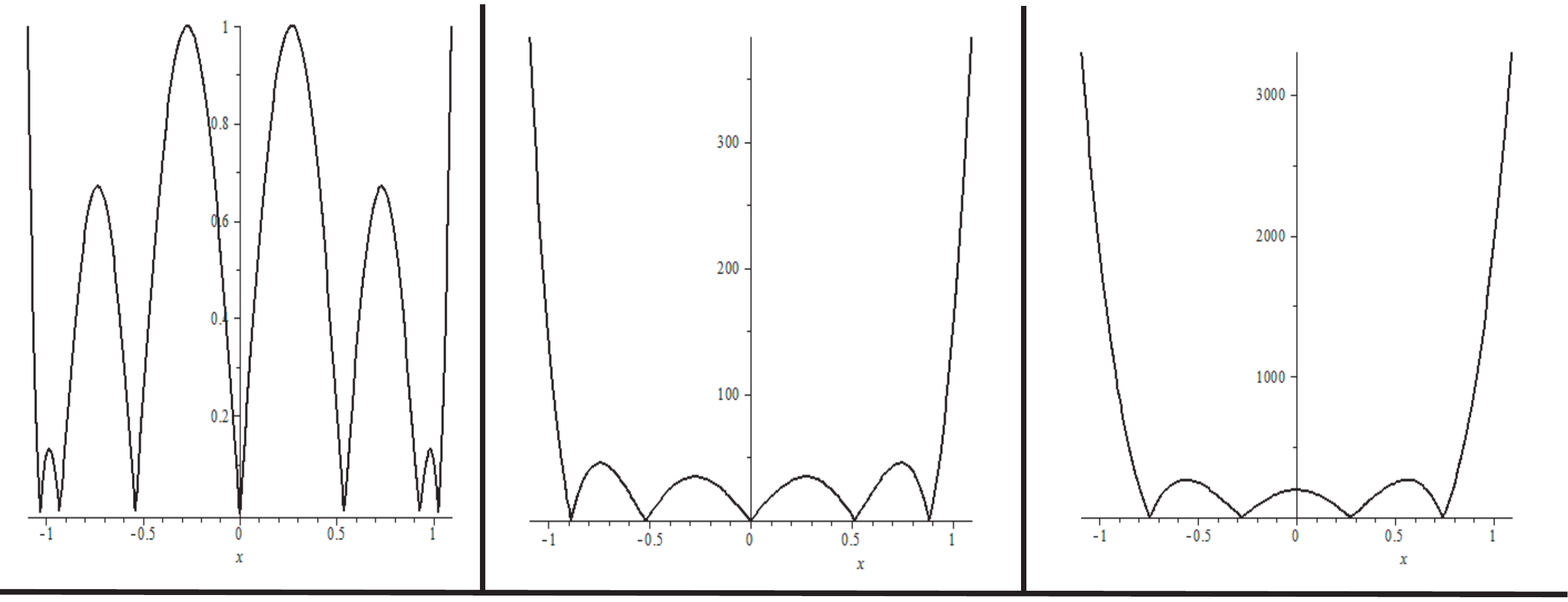}
        \vspace*{-3cm}
   \caption{Graphics of $|\frac{d^k}{dx^k}q^{(\alpha)}_{n,\lambda}|$ for $n=7$, $\alpha=-\frac{1}{4}$, $\lambda=\frac{1}{2}$, $M_{n,\lambda}=1.091516326$ and $k=0,2,3$, respectively.}
\label{figur2}
\end{figure}

\vspace{1cm}

 {\bf Acknowledgments.}\\

The authors would like to express their gratitude to Professors D.K. Dimitrov and G. Nikolov for providing them some helpful academic materials and guidance during the preparation of this manuscript.

\end{document}